\documentclass[12pt]{amsproc}
\usepackage{amssymb}

\usepackage{amsfonts,amsmath,amsthm}
\usepackage{amssymb}
\usepackage[utf8]{inputenc}
\usepackage[english]{babel}
\usepackage{graphicx}
\usepackage{array}
\usepackage{url}
\usepackage{hyperref}
\usepackage{booktabs}

\newtheorem{theorem}{Theorem}[section]
\newtheorem*{theorem*}{Theorem}
\newtheorem{lemma}[theorem]{Lemma}

\newtheorem{definition}[theorem]{Definition}
\theoremstyle{remark}
\newtheorem{remark}[theorem]{Remark}

\newtheorem{example}[theorem]{Example}

\def\R{\mathbb R}
\def\K{\mathbb K}
\def\Z{\mathbb Z}

\def\norma{\|\cdot\|}

\def\XnX{(X,\|\cdot\|_X)}

\def\YnY{(Y,\|\cdot\|_Y)}

\title{Voidness of strict convexity in non-Archimedean fuzzy normed spaces}

\begin{document}

\author
{Javier~Cabello~Sánchez, José~Navarro~Garmendia*}
\address{*Corresponding author: Departamento de Matem\'{a}ticas, Universidad de Extremadura, 
Avenida de Elvas s/n, 06006; Badajoz. Spain}
\email{coco@unex.es, navarrogarmendia@unex.es}

\thanks{Keywords: non-Archimedean fuzzy normed spaces; Mazur--Ulam Theorem; 
non-Archimedean normed spaces; strict convexity; fuzzy 2-normed spaces}
\thanks{Mathematics Subject Classification: 46S10, 26E30}

\begin{abstract} 
In this short note, we show by elementary computations that the notion of 
non-Archimedean fuzzy normed (and 2-normed) spaces is void. Namely, 
there are no strictly convex spaces at all --not even
the zero-dimensional linear space.

Before this, we also study the case of strictly convex non-Archimedean normed 
spaces; in this setting we see that the only nonzero linear space (defined over 
an arbitrary non-Archimedean field) that satisfies this property is the 
one-dimensional linear space over $\mathbb{Z}/3\mathbb{Z}$. 

Consequently, the results that have been proven for this class of spaces, like 
the Mazur-Ulam Theorem, are either trivial or empty statements. 
\end{abstract}

\maketitle

\section{Introduction}

The so called Mazur-Ulam theorem reads as follows:
\begin{theorem*}[Mazur-Ulam]
Let $\XnX$ and $\YnY$ be real normed spaces.  
If $T:X\to Y$ is a bijective isometry 
such that $T(0)=0$ then $T$ is a linear map. 
\end{theorem*}

A first look to the non-Archimedean terrain rapidly shows that the situation 
there is different. Consider some $p$-adic field $\mathbb{Q}_p$ and the map 
$f \colon \mathbb{Q}_p \to \mathbb{Q}_p$ defined as $f(x) := 1/x$, if $|x|=1$, 
and $f(x) = x$ otherwise. It is clear that $f$ is a bijective isometry that maps 
0 to 0 but is certainly not linear (see also \cite[Proposition 1]{KubzdelaUno}). 

The notion of strictly convex non-Archimedean normed space was introduced 
in~\cite{Moslehian1} with the aim of finding a class of spaces where a 
non-Archimedean version of the Mazur-Ulam theorem were true. This paper has 
attracted quite some interest (see, e.g.,~\cite{KangKohCho, KohKang, Kubzdela, Park}), 
probably due to the fact that the authors (apparently) succeeded in proving 
such a result (\cite[Theorem 2.2]{Moslehian1}). 

Nevertheless, the authors did not provide 
any example of a strictly convex non-Archimedean space, and, in fact, A. Kubzdela has 
already pointed out that their existence is ``quite a rarity'' (cf. \cite{KubzdelaUno}, 
Theorem 2 and Remark below). In this note, we prove in Lemma \ref{MenosUnoMasDos} 
that the only non-Archimedean normed space that is strictly convex is the 
one-dimensional linear space over $\mathbb{Z}/3\mathbb{Z}$. Needless to say, 
a version of the Mazur-Ulam theorem trivially holds for this space\ldots\ as 
any bijection (or injection, or surjection) of $\Z/3\Z$ that sends 0 to 0 is linear!

In the same quest for non-Archimedean Mazur-Ulam-type results, other authors 
have also introduced the class of strictly convex fuzzy normed (\cite{KangKohCho}) 
and fuzzy 2-normed spaces (\cite{KohKang}). These notions have been inspired by 
the one given in~\cite{Moslehian1}, and we also prove that there are no spaces 
satisfying their definitions --not even the zero linear space.

\section{Normed spaces}

The first two definitions are a commonplace in non-Archimedean analysis, and we 
reflect them just for the sake of completeness. 

\begin{definition}
A {\em non-Archimedean field} is a field $\K$ equipped with a non-Archimedean 
valuation $|\cdot|$; that is to say, equipped with a function $|\cdot|:\K\to[0,\infty)$ such that 
\begin{enumerate}
\item[i)] $|r | = 0$ if and only if $r = 0$, 
\item[ii)] $|r s| = |r ||s|$, 
\item[iii)] $|r + s| \leq max\{|r |, |s|\}$ for all $r, s \in K$. 
\end{enumerate}
\end{definition}

\begin{remark}\label{trivialZkZ}
Any finite field $\mathbb{F}_q$ is a non-Archimedean field in a unique way: 
equipped with the trivial valuation $|a| = 1$, for any non-zero 
element $a\in \mathbb{F}_q$, and $|0|=0$. 
\end{remark}

\begin{definition}
A {\em non-Archimedean normed space} is a linear space $X$ over a non-Archimedean 
field $(\K , |\cdot|)$ that is endowed with a non-Archimedean norm; that is to 
say, endowed with a function $\norma:X\to[0,\infty)$ such that 
\begin{enumerate}
\item[iv)] $\|x\| = 0$ if and only if $x = 0$, 
\item[v)] $\|r x\| = |r |\|x\|$,  for all $r \in K, x\in X$.     
\item[vi)] $\|x + y\| \leq max\{\|x\|,\|y\|\}$ for all $x,y\in X$.     
\end{enumerate}
\end{definition}

\begin{remark}
The only possible (non-Archimedean) norm on
a one-dimensional linear space over a finite field $\mathbb{F}_q$ is $\|x\|=\lambda>0$ 
for every nonzero vector $x \in \mathbb{F}_q$, and $\|0\|=0$. 
\end{remark}

\begin{remark}\label{remisos} 
On a non-Archimedean normed space $X$, ``any triangle is isosceles"; 
that is to say, for any $x, y \in X$, 
$$ \| y \| < \| x \| \quad \Rightarrow \quad \| x + y \| = \| x \| \ . $$
\end{remark}

\begin{definition}[\cite{Moslehian1}]\label{defnans} 
A non-Archimedean normed space $X$ over a  field $(\K , | \cdot |)$ such that 
$|2|=1$ is said to be {\em strictly convex} if $\|x+y\|=\max\{\|x\|,\|y\|\}$ 
and $\|x\|=\|y\|$ imply $x=y$. 
\end{definition}

In other words, a non-Archimedean normed space $X$ over $\K$ is strictly convex if
\begin{enumerate}
\item[(SC1)] the valuation of $\K$ satisfies $|2|=1$. 
\item[(SC2)] for any pair of vectors $x,y \in X$, $\|x\|=\|y\|=\|x+y\|$ implies $x=y$.
\end{enumerate}

Observe that (SC2) may also be rephrased by saying that there are no equilateral triangles; in other words, that for any pair of distinct vectors $x \neq y \in X$, 
$$ \| y \| = \| x \| \quad \Rightarrow \quad \|x+y \| < \| x \|\ .$$ 

\medskip

\begin{example}\label{Z3Z}
The $\Z /3\Z$-linear space $\Z/3\Z$ is strictly convex. 
\end{example}

Now we may proceed with the only result that actually needs a proof in this paper: 

\begin{lemma}\label{MenosUnoMasDos}
If $X$ is a (non-zero) strictly convex non-Archimedean normed space, 
then $X$ is linearly isometric to $ \Z/3\Z$ endowed with the norm 
$ \| -1 \| = \| 1 \| = \lambda $, for some $\lambda\in(0,\infty)$. 
\end{lemma}

\begin{proof}

As $|2|=1$, for any non-zero vector $x\in X$ we have that
$$\|2x\|=\|x\|=\|-x\|.$$
Of course, $2x+(-x)=x,$ so (SC2) implies that $2x=-x$; in other 
words, $3x =0$ for every $x\in X$, and we conclude that $3=0$ in $\K$. 

Now suppose there are two non-zero vectors $x,y\in X$ such that 
$y\neq \pm x$ and we will arrive to a contradiction.

Without loss of generality we may also assume that 
$$\|y\|\leq\|x\| \quad \mbox{ and } \quad \|x-y\|\leq\|x+y\| \ . $$


If $\|y \|  < \| x \| $, then $x+ y$ and $x-y$ are distinct vectors (by (SC1)) 
such that $\| x + y \| = \| x \| = \| x-y \|$, and whose sum mantains the norm
$$\|(x+y) + (x-y)\| = \|2x\|= \|x\| = \| x + y \|   \ , $$ in contradiction with (SC2).

If $\|y \| = \| x \| $, then (SC2) implies the absurd chain of inequalities
$$ \| x + y \| < \| x \| = \| 2 x \| = \| ( x + y ) + (x-y) \| \leq \| x + y \| \ . $$

\end{proof}

\section{Fuzzy spaces}

It is hard to believe, but things go even worse when we move to fuzzy normed 
spaces: the mistake we find in Definition~\ref{defnans} is not replicated but 
increased. 
To be precise, 

\begin{definition}\label{def021} A {\em non-Archimedean fuzzy normed space} is a linear space $X$ over a non-Archimedean field $(\K , |\cdot|)$ equipped with a non-Archimedean fuzzy norm; that is to say, equipped with a function $N:X\times\R\to[0,1]$ such that, for all $x,y\in X$ and all $s,t\in\R$,
\begin{enumerate}
\item[(N1)] for $t\leq 0,$ $N(x,t)=0$, 
\item[(N2)] for $t>0$, $N(x,t)=1$ if and only if $x=0$,
\item[(N3)] $N(cx,t)=N(x,t/|c|)$ for $c\neq 0$,
\item[(N4)] $N(x+y,\max\{s,t\})\geq \min\{N(x,s),N(y,t)\}$,
\item[(N5)] $N(x,*)$ is a nondecreasing function of $\R$ and $\lim_{t\to\infty}N(x,t)=1$. 
\end{enumerate}
\end{definition}

\begin{definition}[\cite{KangKohCho}]\label{def022}
A non-Archimedean fuzzy normed space $(X,N)$ is said to be strictly convex if 
$N(x+y,\max\{s,t\})=\min\{N(x,s),N(y,t)\}$ and $N(x,s)=N(y,t)$ 
imply $x=y$ and $s=t$.
\end{definition}


But any fuzzy norm $N$ satisfies, 
by (N2), that for any $s, t\in(0,\infty)$, 
$$ N(0,s) = N(0,t) = N(0+0,\max\{s,t\})=1 \ ,$$ 
so no fuzzy normed space $(X,N)$ may fulfil Definition~\ref{def022} --not even the zero linear space.

\begin{definition}\label{def21} A {\em non-Archimedean fuzzy 2-normed space} is a linear space $X$ over a non-Archimedean field $( \K, |\cdot |) $ together with a fuzzy 2-norm; that is to say, with a function $N:X^2\times\R\to[0,1]$ such that, for all $x,y\in X$ and all $s,t\in\R$,
\begin{enumerate}
\item[(2N1)]  for $t\leq 0,$ $N(x,y,t)=0$,
\item[(2N2)] for $t>0$, $N(x,y,t)=1$ if and only if $x$ and $y$ are linearly dependent,
\item[(2N3)] $N(x,y,t)=N(y,x,t),$
\item[(2N4)] $N(x,cy,t)=N(y,x,t/|c|)$ for $c\neq 0$,
\item[(2N5)] $N(x,y+z,\max\{s,t\})\geq \min\{N(x,y,s),N(x,z,t)\}$,
\item[(2N6)] $N(x,y,*)$ is a nondecreasing function of $\R$ and $\lim_{t\to\infty}N(x,y,t)=1$. 
\end{enumerate}
\end{definition}

\begin{definition}[\cite{KohKang}]\label{def22}
A non-Archimedean fuzzy 2-normed space $(X, N)$ is said to be strictly convex if 
$N(x,y+z,\max\{s,t\})=\min\{N(x,y,s),N(x,z,t)\}$ and $N(x,y,s)=N(x,z,t)$ 
imply $y=z$ and $s=t$.
\end{definition}


Again, any fuzzy 2-norm $N$ satisfies, by (2N2), that, for any $x\in X, s, t\in(0,\infty)$, 
\begin{equation}\label{eqn2n}
N(x,x,s)= N(x,2x,t)=N(x,3x,\max\{s,t\})=1, 
\end{equation}
so there is no way that any fuzzy 2-normed space $(X,N)$ fulfils 
Definition~\ref{def22} --not even the 0 space, since we can subtitute $x$ by 
0, $s$ by $1$ and $t$ by $2$ in~(\ref{eqn2n}). Observe, furthermore, that the 
problem with Definition~\ref{defnans} still remains: for any $t\in\R$ and 
$x\in X$ we have $1=N(x,x,t)= N(x,2x,t)=N(x,-x,t)$, so we obtain $3x=0$ 
for any $x\in X$.

\section*{Acknowledgments}
{Supported in part by DGICYT projects MTM2016-76958-C2-1-P and PID2019-103961GB-C21 
(Spain), Feder funds and Junta de Extremadura programs GR-15152, IB-16056 and IB-18087.}


\bibliographystyle{plain}

\end{document}